\documentclass{amsart}
\usepackage{amssymb,amstext,amsmath,amscd,amsthm,amsfonts,enumerate,graphicx,latexsym,stmaryrd,multicol}

\usepackage{hyperref}
\usepackage{cleveref}  
\hypersetup{colorlinks=true} 
\usepackage[usenames]{color}
\usepackage[all]{xy}
\usepackage{tikz-cd}  
\usepackage{tikz} 

\newtheorem{thm}{Theorem}[section]
\newtheorem{lem}[thm]{Lemma}

\newtheorem{cor}[thm]{Corollary}
\theoremstyle{definition}

\newtheorem{chunk}[thm]{}

\theoremstyle{remark}

\numberwithin{equation}{thm}

\def\depth{\operatorname{depth}}

\def\Ext{\operatorname{Ext}}

\def\ge{\geqslant}

\def\Hom{\operatorname{Hom}}

\def\id{\operatorname{id}}

\def\le{\leqslant}

\def\m{\mathfrak{m}}

\def\pd{\operatorname{pd}}

\def\r{\operatorname{r}}

\def\cx{\operatorname{cx}}

\def\syz{\mathsf{\Omega}}

\def\Tr{\operatorname{\mathsf{Tr}}}

\begin{document}
\allowdisplaybreaks
\title[]{When are syzygies of the residue field self-dual?} 
\author[Souvik Dey]{Souvik Dey}
\address{S.~Dey,
Faculty of Mathematics and Physics,
Department of Algebra,
Charles University, 
Sokolovsk\'{a} 83, 186 75 Praha, 
Czech Republic}
\email{souvik.dey@matfyz.cuni.cz}

\subjclass[2020]{13D02, 13H10} 
\keywords{reflexive module, $n$-torsion-free module, syzygy, hypersurface}
\begin{abstract} Finitely generated reflexive modules over commutative Noetherian rings form a key component of Auslander and Bridger’s stable module theory and are likewise essential in the study of Cohen–Macaulay representations. Recently, H. Dao characterized Arf local rings as exactly those one-dimensional Cohen--Macaulay local rings over which every finitely generated reflexive module is self-dual, and raised the general question of characterizing rings  over which every finitely generated reflexive module is self-dual. Motivated by this, in this article, we study the question of self-duality of syzygies of the residue field of a local ring when they are known to be reflexive. We show that for local rings of depth at least 2, the answer is given by hypersurface or regular local rings in most cases. 
\end{abstract}
\maketitle

\section{Introduction}

\textbf{Setup}: Throughout, all rings are assumed to be commutative and Noetherian and all modules are assumed to be finitely generated. For a ring $R$ and an $R$-module $M$, $M^*$ will stand for $\Hom_R(M,R)$. An $R$-module $M$ will be called self-dual if $M\cong M^*$. For a local ring $R$ and an $R$-module $M$, $\syz_R^n M$ will denote the $n$-th syzygy in a minimal $R$-free resolution of $M$.  
\\

Let $R$ be a ring. An $R$-module $M$ is called reflexive if the natural biduality map $M\to M^{**}$ is an isomorphism. In our context of finitely generated modules over commutative Noetherian rings, this is equivalent to the fact that $M\cong M^{**}$ (see \cite[Proposition 1.1.9]{gbook}). Reflexive modules constitute a crucial aspect of Auslander and Bridger’s stable module theory (see \cite{AB69}) and also play a significant role in the theory of Cohen--Macaulay representations (\cite[Section 6]{lw}). Clearly, self-dual modules are reflexive, and then it is natural to ask when the converse holds. Recently, H. Dao showed that among one-dimensional local Cohen--Macaulay rings with reduced completion,  Arf  rings are exactly those over which every reflexive module is self-dual, see \cite[Theorem B]{arf}. Dao then asked whether, in other Krull dimensions, one can characterize rings over which every reflexive module is self-dual, see \cite[Question 4.3]{arf}.

From \cite[Theorem 4.1]{restf}, it is known that for a local ring $(R,\m,k)$ of depth at least $2$, $\syz_R^ik$ is reflexive for every integer $i\ge 2$, and in this article we focus on the question of when these modules are self-dual. Our findings are recorded in the following results.  

\begin{thm}\label{1.1}
Let $(R,\m,k)$ be a local ring of depth $t\geq 3$. Then $\syz^i_R k$ is  reflexive $R$-module for every $2\le i \le t-1$. If, moreover, $\syz^i_R k$ is self-dual for some $2\le i\le t-1$, then $R$ is regular, $t$ is odd, and $t=2i-1$.  
\end{thm}

\begin{thm}[Theorem \ref{5} and Theorem \ref{7}]\label{100} Let $(R,\m,k)$ be a non-regular local ring of depth $t\geq 2$. Then, $\syz^i_R k$ is reflexive for every $i\geq t$. Consider the following conditions 

\begin{enumerate}[\rm(1)]
    \item $(\syz^i_R k)^*\cong \syz^i_R k$ for every $i\geq t$. 

    \item $(\syz^i_R k)^*\cong \syz^i_R k$ for some $i\geq t$. 

    \item $R$ is a hypersurface.   
\end{enumerate}

    Then, $(1)\implies (2)\implies (3)$. If, moreover, $t$ is even, then all three conditions are equivalent.  
\end{thm}

We also make a remark about the depth $1$ Cohen--Macaulay case in \Cref{cmmin}.

We prove both \Cref{1.1} and \Cref{100} in \Cref{main} after recalling some  preliminaries in \Cref{prelim}. 
\\ 

\textbf{Acknowledgments}: The author was partly supported by the Charles University Research Center program No.
UNCE/24/SCI/022 and a grant GACR 23-05148S from the Czech Science Foundation. Almost all of the material presented in this article was developed during the author's final year of graduate studies at the University of Kansas. The author thanks the enriching and supportive environment fostered by the graduate school.  

\section{preliminaries}\label{prelim}

\begin{chunk}
  For an $R$-module $M$, we denote by $\pd_R M$ the projective dimension of $M$. We set $\pd_R 0=-\infty$.   
\end{chunk}

\begin{chunk}\label{para:Tr-M}
    For an $R$-module $M$, consider a projective presentation $F_1 \stackrel{\eta}{\to} F_0 \to M \to 0$ of $M$. The Auslander transpose of $M$ is defined to be $\Tr M := \operatorname{coker}(\eta^*)$ respectively. This module is uniquely determined by $M$ up to projective summands.
    By the definition of $\Tr M$, there is an exact sequence
\begin{equation}\label{eqn:es-M*-Tr-M}
    0 \to M^* \to F_0^* \stackrel{\eta^*}{\longrightarrow} F_1^* \to \Tr M\to 0.
\end{equation}
It is well known that $M$ and $\Tr(\Tr M)$ are stably isomorphic (i.e., these two modules are isomorphic up to projective summands).
Whenever $R$ is local, we define $\Tr M$ by using a minimal free resolution of $M$. It follows from \cite[Proposition 2.2(2)]{crspd} that $\pd_R M^*= \pd_R (\Tr M)-2$  if $\pd_R(\Tr M)\ge 2$. If $\pd_R(\Tr M)\le 1$, then $M^*$ is projective.   
For further details on this topic, see \cite{AB69}. 
\end{chunk}


\begin{chunk}\label{para:-n-torsion-free-nth-syz}
    Let $n\ge 0$ be an integer. Recall from \cite[Defn.~2.15]{AB69} that an $R$-module $M$ is said to be $n$-torsion-free if $\Ext_R^{1\le i \le n}(\Tr M, R)=0$. So, in this terminology, $M$ is called torsionless  if it is $1$-torsion-free and $M$ is reflexive if it is $2$-torsion-free, see, e.g., \cite[1.4.21]{Bruns/Herzog:1998}.  
\end{chunk}

\begin{chunk}\label{cmmin} Let $(R,\m,k)$ be a local ring of depth $1$. Then, $\m$ is reflexive by \cite[Theorem 4.1(2)]{restf}. If $R$ is non-regular, then $\m$ is a regular trace ideal. Thus, when $R$ is one-dimensional local Cohen--Macaulay ring, the (reflexive) $R$-module $\m$ is self-dual if and only if $\m^2=x\m$ for some $R$-regular element $x\in \m$ (see \cite[Proposition 3.3]{daolin}).  If, moreover, $R$ has infinite residue field, then $\m$ is self-dual if and only if $\m^2=x\m$ for some principal reduction $x$ of $\m$ (see \cite[Remark 3.4]{daolin}) if and only if $R$ has minimal multiplicity  (see \cite[4.6.14(c)]{Bruns/Herzog:1998}).  
\end{chunk}   

\section{Main results}\label{main}

\begin{proof}[Proof of \Cref{1.1}] That $\syz^i_R k$ is a reflexive $R$-module for every $2\le i \le t-1$ follows from \cite[Theorem 4.1(1)]{restf}. Now $\pd_R (\syz_R^i k)^*=i-1$ for $1\le i\le t-1$ (see \cite[Proposition 2.2(2), Lemma 2.5]{crspd}). Hence, if $(\syz^i_R k)^*\cong \syz^i_R k$ for some $2\le i\le t-1$, then $\pd_R \syz^i_R k=i-1$, in particular, $R$ is regular. Then, by Auslander-Buchsbaum formula, we have $t-i=\pd_R \syz^i_R k=i-1$ which implies $t=2i-1$.  
\end{proof}  

\begin{cor}\label{cor1} Let $(R,\m,k)$ be a local ring. If every syzygy of $k$ which is reflexive is also self-dual, then $\depth R\le 3$. 
\end{cor}

\begin{proof} Let, if possible, $t:=\depth R\ge 4$. Then, since both $\syz^2_R k$ and $\syz^3_R k$ are reflexive, hence their self-duality (by hypothesis) implies, in view of \Cref{1.1}, that $t=2.2-1=2.3-1$, contradiction! Thus, $t\le 3$. 
\end{proof} 

In view of \Cref{1.1}, we next focus on $\syz^i_R k$ for $i\geq \depth R$ in order to prove \Cref{5}. First, we record some preliminary lemmas  

\begin{lem}\label{22} Let $(R,\m,k)$ be a local ring of depth $t\geq 2$. If $i\geq t$, then $\syz^i_R k$ is reflexive. If, moreover, $\syz^i_R k$ is self-dual for some $i\geq t$, then $R$ is Gorenstein. 
\end{lem}

\begin{proof} That $\syz^i_R k$ is reflexive for every $i\geq t$ follows from \cite[Theorem 4.1(1)]{restf}. Now assume there exists $i\geq t$ such that $\syz^i_R k \cong (\syz^i_R k)^*$. Then, $\syz^i_R k \cong F\oplus \syz_R^{2} \Tr \syz^i_R k$ for some free $R$-module $F$. Since $i\geq t$, so $\syz^i_R k$ is $(t+1)$-torsionfree by \cite[Theorem 4.1(1) and (2)]{restf}, hence at least $3$-torsionfree since $t\geq 2$. Thus, $0=\Ext^3_R(\Tr \syz^i_R k, R)\cong \Ext_R^1(F\oplus \syz_R^{2} \Tr \syz^i_R k, R)$. Hence, $\Ext^1_R(\syz^i_R k, R)=0$, i.e, $\Ext_R^{1+i}(k,R)=0$. Since $i+1\geq t+1$, we get $\id_R R<\infty$ by \cite[II. Theorem 2]{rob}, hence $R$ is Gorenstein.   
\end{proof}

\begin{lem}\label{sum} Let $R$ be a local ring and $M,N$ be finitely generated modules without free summands. If $M,N$ are stably isomorphic i.e. $M\oplus F_0 \cong N \oplus F_1$ for some finitely generated free $R$-modules $F_0,F_1$, then it holds that $M\cong N$.  
\end{lem}

\begin{proof}  Let $M\oplus R^{\oplus a}\cong N \oplus R^{\oplus b}$. If $a>b$, then by \cite[Corollary 1.16]{lw}, we would have $M\oplus R^{\oplus (a-b)}\cong N$, and then \cite[Lemma 1.2(i)]{lw} implies $R$ is a summand of $N$, contradicting the assumption on $N$. Thus we must have $a\le b$. Similarly $b\le a$. Thus $a=b$, and so $M\oplus R^{\oplus a}\cong N \oplus R^{\oplus a}$, and then $M\cong N$ by \cite[Corollary 1.16]{lw}.  
\end{proof} 

\begin{lem}\label{3} Let $(R,\m,k)$ be a local ring of depth $t\geq 1$. Let $\r :=\dim_k\Ext^t_R(k,R)$. Then, the following hold true

\begin{enumerate}[\rm(1)]
    \item We have an isomorphism $F\oplus \syz^t_R\left((\syz^t_R k)^*\right)\cong \syz^t_R k^{\oplus \r}$ for some free $R$-module $F$.

    \item Let $n\geq 1$ be an integer and $M$ be an $R$-module such that $\Ext_R^{1\le i\le n}(M,R)=0$. Then, $M^* \cong G\oplus \syz^n_R\left((\syz^n_R M)^*\right)$ for some free $R$-module $G$. 

    \item Assume $R$ is Gorenstein. Then, for every $i\geq t$, there exists a free $R$-module $H$ such that $ \syz^t_R k \cong  H\oplus \syz^i_R\left((\syz^i_R k)^*\right)$. If $R$ is Gorenstein and non-regular, then $ \syz^t_R k \cong \syz^i_R\left((\syz^i_R k)^*\right)$ for every $i\geq t$.    
\end{enumerate} 
\end{lem}

\begin{proof} (1) There exists a free $R$-module $F_0$ and an exact sequence $0\to \syz^t_R k \to F_0\to \syz_R^{t-1} k\to 0$ dualizing which with respect to $R$ we get an exact sequence $0\to (\syz^{t-1}_R k)^*\to F_0\to (\syz^t_R k)^*\to k^{\oplus \r}\to 0$. Split this into exact sequences $0\to (\syz^{t-1}_R k)^*\to F_0\to N\to 0$ and $0\to N \to (\syz^t_R k)^*\to k^{\oplus \r}\to 0$. Now, if $t=1$, then $(\syz^{t-1}_R k)^*=0$, and so $N$ is free, hence $\syz_R N=0$. If $t\ge 2$, then $\pd_R (\syz_R^{t-1} k)^* \leq t-2$ (see \cite[Proposition 2.2(2), Lemma 2.5]{crspd}), hence $\pd_R N\leq t-1$, so $\syz^t_R N=0$. In either case, $\syz^t_R N=0$. Now applying \cite[Proposition 2.2(1)]{crspd} to $0\to N \to (\syz^t_R k)^*\to k^{\oplus \r}\to 0$ we get $F\oplus \syz^t_R\left((\syz^t_R k)^*\right)\cong \syz^t_R k^{\oplus \r}$ for some free $R$-module $F$. 

(2) We have an exact sequence $$0\to \syz^n_R M\to F_{n-1}\to\cdots\to F_0\to M\to 0$$ for some free $R$-modules $F_i$. Dualizing this and remembering $\Ext_R^{1\le i\le n}(M,R)=0$, we get an exact sequence $0\to M^*\to F_0\to\cdots\to F_{n-1}\to (\syz^n_R M)^*\to 0$  which proves our claim. 

(3) Assume $R$ is Gorenstein and fix $i\geq t$. Then, $\Ext_R^{>0}(\syz^t_R k , R)=0$. Applying part (2) with $M=\syz^t_R k$ and $n=i-t$, we get $(\syz^t_R k)^* \cong G\oplus \syz^{i-t}_R\left((\syz^{i}_R k)^*\right)$ for some free $R$-module $G$. Applying $\syz^t_R(-)$ to both sides, we get $\syz^t_R\left((\syz^t_R k)^*\right)\cong \syz^i_R\left((\syz^i_R k)^*\right)$. Now applying part (1) and recalling $\r=1$ we get $\syz^t_R k\cong F \oplus \syz^i_R\left((\syz^i_R k)^*\right)$ for some free $R$-module $F$. 

Now assume, moreover, that $R$ is not regular. Then, $\syz^i_R k$ has no free summand for every $i\geq 0$ (\cite[ Corollary 1.3]{rd}). Moreover, $\syz^i_R\left((\syz^i_R k)^*\right)\cong \syz_R^{i+2} \Tr \syz^i_R k$ does not have a free summand for every $i\geq t$ by \cite[Lemma 2.2]{toshi}. Now, the required isomorphism follows by Lemma \ref{sum}.
\end{proof}  

Now we strengthen \Cref{22} by proving $(2)\implies (3)$ of \Cref{100}.  We recall that a local ring $(R,\m)$ is called a hypersurface if $\mu(\m)-\depth(R)\le 1$ (see \cite[5.1]{av}). 

\begin{thm}\label{5} Let $(R,\m,k)$ be a local ring of depth $t\geq 2$. Then, $\syz^i_R k$ is reflexive for every $i\geq t$. If, moreover, $\syz^i_R k$ is self-dual for some $i\geq t$, then $R$ is a hypersurface. 
\end{thm}

\begin{proof} That $\syz^i_R k$ is reflexive for every $i\geq t$ follows from \cite[Theorem 4.1(1)]{restf}. If, moreover, $\syz^i_R k$ is self-dual for some $i\geq t$, then by Lemma \ref{22} and Lemma \ref{3}(3) we get $\syz^t_R k \cong H\oplus \syz_R^{2i} k$ for some free $R$-module $H$, thus $\syz^{t+1}_R k \cong \syz^{2i+1}_R k$. Since $2i>t$, thus $k$ has bounded Betti sequence, hence $\cx_R(k)\leq 1$, and thus, $R$ is a hypersurface (see \cite[Theorem 8.1.2]{av}). 
\end{proof}  

Now to prove kind of a converse to \Cref{5}, we notice the following lemma  

\begin{lem}\label{tot} Let $M,N$ be finitely generated  modules over a local ring $R$. If there exists an integer $n\ge 1$ such that such that $\Ext^{1\le i\le n }_R(M,R)=\Ext^{1\le i \le n}_R(N,R)=0$ and  $\syz^n M$ is stably isomorphic with $\syz^n N$, then $M$ and $N$ are stably isomorphic.    
\end{lem}

\begin{proof} If $n\geq 2$ and $\Ext^{1\le i\le n }_R(M,R)=\Ext^{1\le i \le n}_R(N,R)=0$, then $\Ext^{1\le i\le n-1 }_R(\syz_R M\oplus F,R)=\Ext^{1\le i \le n-1}_R(\syz_R N\oplus G,R)=0$ for every free $R$-module $F,G$. Hence, by induction, it is enough to prove the $n=1$ case, i.e.,  if $\Ext^{1}_R(M,R)=\Ext^{1}_R(N,R)=0$ and $\syz M$ is stably isomorphic with $\syz N$, then $M$ and $N$ are stably isomorphic.  So let $X:=\syz M \oplus F \cong \syz N \oplus G$, where $F,G$ are free $R$-modules. Then there exists finitely generated free $R$-modules $F_0$ and $F_1$ which fits into exact sequences $0\to X \to F_0 \to M \oplus F \to 0$ and $0\to X \to F_1 \to N \oplus G\to 0$. This gives rise to the following pushout diagram

\begin{tikzcd}
            & 0 \arrow[d]                                & 0 \arrow[d]                &                                             &   \\
0 \arrow[r] & X \arrow[r] \arrow[d]                      & F_0    \arrow[r] \arrow[d] & M \oplus F  \arrow[r] \arrow[d, equal] & 0 \\
0 \arrow[r] & F_1 \arrow[d] \arrow[r]                    & Y \arrow[r] \arrow[d]      & M \oplus F \arrow[r]                                 & 0 \\
            & N\oplus G \arrow[d] \arrow[r,  equal] & N \oplus G \arrow[d]                &                                             &   \\
            & 0                                          & 0                          &                                             &  
\end{tikzcd}
 
 Since $\Ext^1_R(M\oplus F,R)=\Ext^1_R(N\oplus G,R)=0$, so both the middle column and the middle row splits, thus giving us $Y\cong M \oplus F_1\oplus F\cong N\oplus G\oplus F_0$. This concludes the proof.    
\end{proof}   

Finally, we have a  strong partial converse to \Cref{5}. 

\begin{thm}\label{7} Let $(R,\m,k)$ be a local non-regular hypersurface of even dimension $t>0$. Then, $(\syz^i_R k)^*\cong \syz^i_R k$ for every $i\geq t$. 
\end{thm}

\begin{proof} From Lemma \ref{3} we have $ \syz^t_R k \cong \syz^i_R\left((\syz^i_R k)^*\right)$ for every $i\geq t$. Since $R$ is not regular so $R$ is not a summand of $\syz^i_R k$ (see \cite[Corollary 1.3]{rd}). Since $t$ is even, so $2i-t$ is also even, hence $R$ being a hypersurface implies $\syz_R^t k\cong \syz_R^{t+2i-t}\cong \syz_R^{2i} k$ (see \cite[Theorem 5.1.1.]{av}). Thus, $\syz_R^{2i} k\cong \syz^i_R\left((\syz^i_R k)^*\right)$. Since $\Ext_R^{>0}(\syz^i_R k, R)=0=\Ext_R^{>0}((\syz^i_R k)^*, R)$, hence Lemma \ref{tot} gives $\syz^i_R k$ and $(\syz^i_R k)^*$ are stably isomorphic. Since $R$ is not regular so $R$ is not a summand of $\syz^i_R k$ (see \cite[ Corollary 1.3]{rd}), and so $R$ is neither a summand of $(\syz^i_R k)^*$ (as $\syz^i_R k$ is reflexive). Thus, Lemma \ref{sum} yields $(\syz^i_R k)^*\cong \syz^i_R k$. 
    
\end{proof}

\begin{cor} Let $R$ be a local ring of depth at least $2$. If every syzygy of residue field of $R$ which is reflexive is also self-dual, then either $R$ is a regular local ring of dimension $3$, or $R$ is a hypersurface of dimension $2$. 
\end{cor}

\begin{proof} This follows from \Cref{1.1}, \Cref{cor1} and \Cref{100}.
\end{proof}

\bibliographystyle{plain}
\bibliography{mainbib} 

\begin{thebibliography}{10}

\bibitem{AB69}
M.~Auslander and M.~Bridger.
\newblock {\em Stable module theory}, volume~94 of {\em Mem. Am. Math. Soc.}
\newblock Providence, RI: American Mathematical Society (AMS), 1969.

\bibitem{av}
Luchezar~L. Avramov.
\newblock Infinite free resolutions.
\newblock In {\em Six lectures on commutative algebra. Lectures presented at the summer school, Bellaterra, Spain, July 16--26, 1996}, pages 1--118. Basel: Birkh{\"a}user, 1998.

\bibitem{Bruns/Herzog:1998}
Winfried Bruns and J\"{u}rgen Herzog.
\newblock {\em Cohen-{M}acaulay rings}, volume~39 of {\em Cambridge Studies in Advanced Mathematics}.
\newblock Cambridge University Press, Cambridge, 1998.

\bibitem{arf}
Hailong Dao.
\newblock Reflexive modules, self-dual modules and arf rings, https://arxiv.org/abs/2105.12240, 2021.

\bibitem{daolin}
Hailong Dao and Haydee Lindo.
\newblock Stable trace ideals and applications.
\newblock {\em Collect. Math.}, 75(2):395--407, 2024.

\bibitem{crspd}
Hailong Dao and Ryo Takahashi.
\newblock Classification of resolving subcategories and grade consistent functions.
\newblock {\em Int. Math. Res. Not.}, 2015(1):119--149, 2015.

\bibitem{restf}
Souvik Dey and Ryo Takahashi.
\newblock On the subcategories of {{\(n\)}}-torsionfree modules and related modules.
\newblock {\em Collect. Math.}, 74(1):113--132, 2023.

\bibitem{rd}
S.~P. Dutta.
\newblock Syzygies and homological conjectures.
\newblock Commutative algebra, {Proc}. {Microprogram}, {Berkeley}/{CA} ({USA}) 1989, {Publ}., {Math}. {Sci}. {Res}. {Inst}. 15, 139-156 (1989)., 1989.

\bibitem{toshi}
Toshinori Kobayashi.
\newblock On delta invariants and indices of ideals.
\newblock {\em J. Math. Soc. Japan}, 71(2):589--597, 2019.

\bibitem{lw}
Graham~J. Leuschke and Roger Wiegand.
\newblock {\em Cohen-{M}acaulay representations}, volume 181 of {\em Mathematical Surveys and Monographs}.
\newblock American Mathematical Society, Providence, RI, 2012.

\bibitem{rob}
Paul Roberts.
\newblock Two applications of dualizing complexes over local rings.
\newblock {\em Ann. Sci. {\'E}c. Norm. Sup{\'e}r. (4)}, 9:103--106, 1976.

\bibitem{gbook}
Lars Winther~Christensen.
\newblock {\em Gorenstein dimensions}, volume 1747 of {\em Lect. Notes Math.}
\newblock Berlin: Springer, 2000.

\end{thebibliography}

\end{document}